\documentclass[psamsfonts]{amsart}

\address{Heng Xie, School of Mathematics, Sun-Yat Sen University, Guangzhou, China}
\email{xieh59@mail.sysu.edu.cn}

\usepackage{array}
\usepackage{booktabs}
\usepackage{multirow}
\usepackage{amssymb}
\usepackage{amsmath}
\usepackage{amsthm}
\usepackage{amsmath,amscd}
\usepackage{dsfont}
\usepackage{float}
\usepackage{mathrsfs}
\usepackage{amsmath,amscd}
\usepackage{calligra,mathrsfs}
\usepackage{MnSymbol}
\usepackage[all,cmtip]{xy}
\usepackage{enumerate}
\usepackage[graphicx]{realboxes}
\usepackage{hyperref}
\theoremstyle{plain}
\newtheorem{theo}{Theorem}[section]
\usepackage{longtable}
\usepackage{lscape}
\usepackage{enumitem}

\theoremstyle{theorem}
\newtheorem{lem}[theo]{Lemma}
\newtheorem{Prop}[theo]{\bf Proposition}

\theoremstyle{definition}
\newtheorem{remark}[theo]{Remark}

\newtheorem{Def}[theo]{Definition}
\newtheorem*{ackno}{Acknowledgement}
\newtheorem{question}[theo]{\bf Question}

\newcommand{\C}{\mathrm{C}}

\newcommand*{\h}{\mathcal{H}\kern -.5pt om}

\newcommand{\SO}{\mathcal{O}}

\newcommand{\UC}{\mathscr{U}}

\newcommand{\R}{\mathbb{R}}
\newcommand{\Z}{\mathbb{Z}}

\newcommand{\id}{\textnormal{id}}

\newcommand{\tr}{\textnormal{Tr}}
\newcommand{\Hom}{\textnormal{Hom}}

\newcommand{\Spec}{\textnormal{Spec\,}}

\newcommand{\Ker}{\textnormal{ker}}

\newcommand{\W}{W}

\title{The Witt ring of the real sphere}

\author{Heng Xie}

\begin{document}
	
	\begin{abstract}
We calculate the Witt ring of the real sphere.
	\end{abstract}
	
	\maketitle

	\raggedbottom
	\section{Introduction}
 
 In the 1970s, 	Knebusch introduced the Witt ring $\W(X)$ of quadratic forms over a scheme $X$ and pointed out that the Witt ring $\W(S^n)$ remained uncomputed (cf. \cite[p.\ 105]{Kne77}), where  
	\[S^n:=  \Spec\, \mathbb{R}[x_0, \ldots, x_n]/(\sum_{i=0}^n x_i^2 -1) \]
	denotes the real sphere. Nearly fifty years have passed, it might be surprising to the non-specialist that, even the group structure of the Witt ring of such an elementary variety has not been fully determined. This is striking since the analogous computations for $K$-theory and Chow groups are well-known. The aim of this article is to solve this problem of computing the ring $\W(S^n)$. 
	
	\begin{theo}\label{thm:main}\label{thm:main}  
		The underlying group of $\W(S^n)$ is computed as follows. 
\begin{table}[htbp]
\Small 
	\centering
	\renewcommand{\arraystretch}{1.2}
	\begin{tabular}{|c|c|c|c|c|c|c|c|c|c|c|}
		\hline
		$n \mod 8$ & $1$ & $2$ & $3$ & $4$ & $5$ & $6$ & $7$ & $8$ \\
		\hline \hline
		$\W(S^n)$ & $\mathbb{Z} \oplus \mathbb{Z}_2$ & $\mathbb{Z} \oplus \mathbb{Z}_2$ & $ \quad \mathbb{Z} \quad $  & $\mathbb{Z} \oplus \mathbb{Z} $ & $ \quad \mathbb{Z} \quad $ & $ \quad \mathbb{Z} \quad $ & $ \quad \mathbb{Z} \quad $ & $\mathbb{Z} \oplus \mathbb{Z}$ \\
		\hline
	\end{tabular}
\end{table}
	\end{theo}
		In the 1990s, Knus \cite{Knus} made the computation of $\W(S^n)$ for $n=1,2$ (See also \cite[Example 1.2.6]{Ba05} for a result of Ojanguren on this).  In the 1980s, Brumfiel \cite{Brumfiel} constructed the signature map (a ring homomorphism)
			  \[ \gamma: \W(X) \to KO(X_\R)\]
			  for any algebraic variety over $\mathbb{R}$ and its associate space $X_\R$ of $\R$-points, and showed $\gamma$ is an isomorphism away from two torsions, i.e. $\gamma: \W(X)[\frac{1}{2}] \xrightarrow{\cong} KO(X_\R)[\frac{1}{2}]$. More recently, another proof of Brumfiel's theorem was given by Karoubi, Schlichting, and Weibel \cite{KSW16}. In 2007, Dell'Ambrigio and Fasel computed $\W(S^n)[\frac{1}{2}]$ via Balmer's localization sequence away from two, which is independent of Brumfiel, (cf. \cite{DF08}). 
			 
			  Unfortunately, none of these results reveal the two torsion information of $\W(S^n)$, except for $n = 1,2$. Therefore, the two torsion problem for $\W(S^n)$ ($n\geq 3$) is exactly what Theorem \ref{thm:main} resolves.\ Experts immediately see that $\W(S^n)$ has the pattern of the classical 8-Bott periodicity, i.e. share the same pattern with Atiyah's $KO$-theory $KO(S^n_\R) \cong \mathbb{Z} \oplus \pi_n(O^\infty)$. This leads us to show that 
			   					
\begin{theo}\label{thm:solution-to-Knebusch}
The Brumfiel's map $\gamma: \W(S^n) \to KO(S^n_\R)$ is a ring isomorphism. 
\end{theo}
Since the ring structure of $KO(S^n_\R)$ is well-known, Theorem \ref{thm:solution-to-Knebusch} settles the problem of computing the ring $\W(S^n)$. 
 Our search for explicit generators of $W(S^n)$ leads us to  a condensed version of  the Clifford-module description of the non-trivial generator of $KO(S^n_\R)$ due to Atiyah-Bott-Shapiro \cite{ABS63} and Fossum \cite{Fossum69} (see Formula \eqref{eq:Phi} and Definition \ref{def:Pn}, and compare with $\ker(\alpha_M)$ in \cite{Fossum69}). 
Together with Swan's result \cite[Theorem 3]{Swan}, we see that Witt-theory, $KO$-theory and $K$-theory are unified in the case of real spheres, i.e.
\[ \W(S^n)  \xrightarrow{\cong} KO(S^{n}_\R) \xleftarrow{\cong}  K_0(S^n). \]

It is commonly believed that Brumfiel's map $\gamma$ is not an isomorphism integrally (neither injective nor surjective), but only becomes an isomorphism away from two torsions, since there could be real vector bundles on $X_\R$ that can not be realised as an algebraic quadratic form over $X$. This is the case for projective spaces $\mathbb{P}^n$. It might be surprising that Theorem \ref{thm:solution-to-Knebusch} shows that quadratic forms over the real spheres behave nicely, so that all the quadratic forms (up to hyperbolic forms) over real sphere $S^n$ precisely comes from topological real vector bundles stably. It leads to the following question:
\begin{question}\label{ques:Brumfiel}
For which real algebraic varieties is Brumfiel's map $\gamma$ an  isomorphism?
\end{question}
For $X$ to be a point, Question \ref{ques:Brumfiel} has a trivial positive answer. Our real sphere computation provides an interesting non-trivial example to the positive answer of Question \ref{ques:Brumfiel}. In general, the answer of Question \ref{ques:Brumfiel} might depend on the development of the computation of Witt rings of algebraic varieties, and measures the gap between the algebraic theory of quadratic forms and the topology of real vector bundles.

\noindent\textbf{Convention} The quotient group $\Z / m \Z$ is denoted by $\Z_m $ throughout the paper. For any integer $n$, let $\bar{n}$ denote the unique integer such that $1 \leq \bar{n} \leq 8$ and $n \equiv \bar{n} \mod 8$. 
	\section{Proof strategy}
	The Witt ring of quadratic forms over a field was introduced by Witt in 1937, cf.\  \cite{Witt37}. There are several well-known books on the subject, for instance those by  Knus \cite{Knus}, Lam \cite{Lam}, Milnor and Husemoller \cite{MH73}, and Scharlau \cite{Scharlau}. For a useful survey of the framework needed here, see \cite{Ba05}.

\subsection{Construct the non-trivial generator} We will show that the non-trivial generator of $\W(S^n)$ occurs only when $n \equiv 1,2,4,8 \mod 8$ in Section \ref{sec:Wittsphere}. The aim of this section is to construct them explicitly. For Clifford algebras, we follow the notation $\C^{p,q}$ used in Karoubi's book \cite[Section 3]{Karoubi78} (See Section \ref{sec:Clifford} for a summary).

\subsubsection{ } Let $A_n $ be the coordinate ring $\mathbb{R}[x_0, \ldots, x_n]/(\sum_{i} x_i^2 -1)$. Take the generators $e_0,\ldots, e_n \in \C^{0,n+1}$ with $e_i^2 =1$ and $e_i e_j = - e_j e_i$ for $i\neq j$. Consider the endomorphism 
\[ \Phi: A_n \otimes_{\mathbb{R}} \C^{0,n+1} \to A_n \otimes_{\mathbb{R}} \C^{0,n+1} \]
given by 
\begin{equation}\label{eq:Phi} \Phi(a \otimes u) = \frac{1}{2}( \sum_{i=0}^n x_i  a \otimes e_i  u + a \otimes u) =: \frac{1}{2}(\mathbf{x} +\mathbf{1}) \cdot  (a\otimes u)    \end{equation}
where $\mathbf{x} := \sum_{i=0}^n x_i \otimes e_i$ is the left multiplication.\ It is not hard to see that $\Phi$ is an idempotent operator, since $\sum_ix_i^2 =1 $ in $A_n$. Let $\sigma:\C^{0,n+1} \to \C^{0,n+1}$ be the canonical involution, i.e. $\sigma(e_i) = e_i$ and $\sigma(xy) = \sigma(yx)$.\ Note that the trace form 
\[ \beta: \C^{0,n+1} \times \C^{0,n+1} \to \mathbb{R}, (x,y) \mapsto  \tr(\sigma(x) y) \] 
is positive definite and non-degenerate. Moreover, the trace form $\beta$ is compatible with $\Phi$, i.e. the diagram
\[ \xymatrix{ A_n \otimes_{\mathbb{R}} \C^{0,n+1} \ar[r]^-{\Phi} \ar[d]^-{\hat\beta} &  A_n \otimes_{\mathbb{R}} \C^{0,n+1} \ar[d]^-{\hat\beta}  \\
                  \Hom_{A}( A_n \otimes_{\mathbb{R}} \C^{0,n+1} , A_n ) \ar[r]^{\Phi^\vee} &  \Hom_{A}( A_n \otimes_{\mathbb{R}} \C^{0,n+1} , A_n )           }\]
commutes where $\hat\beta(a \otimes x, b \otimes y) := ab \beta(x,y)$. In other words, $\hat{\beta}(\Phi(x) , y) =  \hat{\beta}(x , \Phi(y)) $. 

\subsubsection{ }
The algebra $\C^{0,n+1}$ is semi-simple, and we can always find a primitive idempotent element $\wp$ so that $\C^{0,n+1} \wp$ is an irreducible left ideal in $\C^{0,n+1}$.   

\begin{Def}\label{def:Pn}
Define \[P_n:= \ker(\Phi: A_n \otimes_{\mathbb{R}} \C^{0,n+1} \wp \to A_n \otimes_{\mathbb{R}} \C^{0,n+1} \wp ). \]
\end{Def}

Since $\Phi$ is an idempotent operator, $P_n$ is a projective $A_n$-module.  

\begin{lem}\label{lem:Pntheta}
The form $\hat\beta$ restricting to $P_n$ is non-degenerate. 
\end{lem}
\begin{proof}
The trace form 
\[ \beta: \C^{0,n+1} \times \C^{0,n+1} \to \mathbb{R}, (x,y) \mapsto  \tr(\sigma(x) y) \] 
is positive definite. Hence, it restricts a non-degenerate symmetric form  $ \beta: E \times E\to \mathbb{R}$ on the subspace $E: = \C^{0,n+1}\wp$, which gives rise to a non-degenerate form on $A_n\otimes_\R E$ by tensoring the trivial form on $A_n$. Write
\[A_n \otimes_\R E = \ker(\Phi) \oplus \ker(1 - \Phi) = P_n \oplus Q_n \]
Since $\hat{\beta}(\Phi(x) , y) =  \hat{\beta}(x , \Phi(y)) $, we must have $\hat{\beta}(p,q) = 0$ for any $p \in P_n$ and $q \in Q_n$. It follows that $\hat{\beta}$ restricts to a non-degenerate form on $P_n$. 
\end{proof}
\begin{remark}
The module $P_n$ provides another construction of the projective module over $A_n$ studied by Fossum \cite{Fossum69}. Later, we will show that the form $(P_n, \hat\beta)$ represents the non-trivial generator in $\W(S^n)$ for $n \equiv 1,2,4,8 \mod 8$. 
\end{remark}

\subsection{On the Brumfiel's map} Let $X$ be a smooth algebraic variety over $\R$. We may assume that $X_\R$ is compact which is enough for us. Recall the Brumfiel's map 
\begin{align*}
	\gamma:  \W(X) &\to KO(X_\R) \\
	[V,\phi] & \mapsto [V_\R^+] - [V_\R^-]
\end{align*} 
where the form $(V_\R,\phi_\R)$ on $X_\R$ yields a decomposition $V_\R = V_\R^+ \oplus V_\R^-$ such that the restriction of form $\phi_\R$ to $V_\R^+$ (resp. $V_\R^-$) is positive (resp. negative), cf. \cite[Chapter I.8]{Karoubi78}. 
The aim of this section is to prove that 
\begin{Prop}\label{prop:brumfiel-surj}
	The map $\gamma: \W(S^n) \to KO(S^n_{\R})$ is surjective. 
\end{Prop}
\begin{proof}
	If $n \equiv 3,5,6,7 \mod 8$, then $KO(S_{\R}^n)$ is isomorphic to $\Z$ generated by the trivial line bundle $S^n\times \mathbb{R}$. It is clear that $\gamma$ sends the trivial form $(\mathcal{O}, \id)$ of $S^n$ to $S_{\R}^n\times \mathbb{R}$.  The result follows. 
	
	If $n \equiv 1,2,4,8 \mod 8$, then $KO(S_{\R}^n)$ has two generators. One is represented by the trivial bundle $S_{\R}^n\times \R$, and $\gamma[\SO,\id] = [S_{\R}^n\times \R]$ as before. The other one $\xi$ which is non-trivial can be however desribed by an algebraic process (cf.\ Atiyah-Bott-Shapiro \cite{ABS63}). Fossum \cite{Fossum69} further studied this program which is what we use here. 

Let $P:= P_n$ and $E:=\C^{0,n+1}\wp$.\ Since $e_0 \in \C^{0,n+1}$ satisfies $e_0^2 =1$, we have a decomposition 
\[ E = E^0 \oplus E^1\]
with $ E^0 := \ker(e_0+1)$ and  $E^1 := \ker(e_0-1)$. 
Set $f_i = e_0 e_i$ (for $i \neq 0$) which serves as the generators of $\C^{n,0}$. It is clear that $f_i^2 = -1$ and $f_i f_j = - f_j f_i$ for $i \neq j$. If $(u^0,u^1) \in E^0 \oplus E^1$, then $f_iu^0 \in E^1$, since $e_0 f_i u^0 = - f_i e_0 u^0 = f_i u^0 $, and similarly $f_iu^1 \in E^0$. Therefore, $E= E^0 \oplus E^1$ is a graded left $\C^{n,0}$-module in view of the basis $f_i$. The irreducibility of $E$ as a left $\C^{0,n+1}$ implies that $E$ is a graded left irreducible $\C^{n,0}$-module. 

The form $\beta_\R$ is positive definite, i.e. a metric on $P_\R$, therefore $\gamma([P,\beta]) = [P_\R]$. Note that, via realization, the operator \eqref{eq:Phi} induces a map $ S_{\R}^n\times  E \xrightarrow{\Phi_\R} S_{\R}^n \times  E $ by 
\[ \Phi_\R (\mathbf{a}, u) =(\mathbf{a}, \frac{1}{2} (\mathbf{1}+ \mathbf{a_e}) u ) \]
where $\mathbf{a} := (a_0 , \ldots, a_n) \in \mathbb{R}^{n+1}$ and $\mathbf{a_e} := \sum_{i =0}^n a_i e_i$ such that $\sum_{i =0}^n a_i^2 = 1$.\ Hence, $P_\R = \ker (\Phi_\R)$.\  
It suffices to show that $\Phi_\R $ coincide with $ \alpha_E$ in \cite{Fossum69}. Note that 
\[ e_0 u^0 = -u^0, \quad e_0 u^1 = u^1, \quad e_i u^0 = f_i u^0, \quad e_i u^1 = - f_i u ^1   \]
for $i \geq 1$. Let $a^*$ denote $\sum_{i=1}^n a_i f_i$ as in \cite{Fossum69}. The map $\Phi_\R$ then acts as
\begin{align*}
	\Phi_\R(\mathbf{a}, u^0 + u^1) &=  \big( \mathbf{a}, \frac{1}{2}(\mathbf{1} + \mathbf{a_e})(u^0 + u^1) \big) \\
	&= \big( \mathbf{a}, \frac{1}{2} ( u^0 + u^1 - a_0 u^0 + a_0 u^1 + a^* u^0 - a^* u^1 ) \big) 
\end{align*}
which coincide with Fossum's $\alpha_E(\mathbf{a}, u^0 + u ^1)$ via the automorphism 
\[T:  S_{\R}^n\times  E \xrightarrow{\Phi} S_{\R}^n \times  E, \quad (\mathbf{a}, u^0 + u^1) \mapsto (\mathbf{a}, u^0 - u^1).   \]   
Hence, $P_\R \cong \ker(\alpha_E)$ which is the non-trivial generator of $KO(S^n)$. 
\end{proof}

\subsection{Proof of Theorem  \ref{thm:solution-to-Knebusch} }
In view of Proposition \ref{prop:brumfiel-surj}, Theorem \ref{thm:solution-to-Knebusch} will follow from Theorem \ref{thm:main} by noting that $\W(S^n)$ and $KO(S^n)$ are both isomorphic to the same finitely generated abelian group. Thus, we shall focus on proving Theorem \ref{thm:main}.

\subsection{Proof outline of Theorem \ref{thm:main}} 
The full proof of Theorem \ref{thm:main} will be given in Section \ref{sec:Wittsphere}. This section aims to outline the proof strategy for the convenience of our readers. 
Our proof of Theorem \ref{thm:main} relies on Balmer's Witt theory \cite{Ba00} together with the author's computations of Witt groups of quadric hypersurfaces (cf. \cite{X19} and \cite{X26}).\ Our argument in proving Theorem \ref{thm:main} is entirely algebro-geometric and does not appeal to $KO$-theory. Following Dell'Ambrogio and Fasel \cite{DF08}, we set
\begin{align*}
\textnormal{	$\mathbb{P}S^n  =   \mathrm{Proj}\,\mathbb{R}[x_0,  \ldots, x_n,y]/(\sum\nolimits_i x_i^2 -y^2),\quad 
	\mathbb{P}C^{n}  = \mathrm{Proj}\,\mathbb{R}[x_0, \ldots, x_n]/( \sum\nolimits_i x_i^2 )$. }
\end{align*}
Let us consider the diagram
\begin{equation*}
	\small
	\xymatrix{ 
		\mathbb{P}C^{n} \ar[r]^-{i} \ar[d]_{\iota} & 	\mathbb{P}S^{n}  \ar[d]_{\kappa} & \ar[l]_-{j}  S^{n} \ar[d]_{\gamma} \\
		\mathbb{P}^{n} \ar[r]^-{\imath} &  \mathbb{P}^{n+1} & \ar[l]_{\jmath} \mathbb{A}^{n+1}
	}
\end{equation*}
in which $i,\iota,\imath,\kappa,\gamma$ are closed embeddings and $j,\jmath$ are open embeddings. Let $\SO(i)$ be the line bundle $\kappa^*\SO_{\mathbb{P}^{n+1}}(i)$ on $\mathbb{P}S^n$. Then, $j^*\SO(i)$ is trivial over $S^n$ for any $i$. 

Since the real sphere $S^n$ is isomorphic to the open subscheme $D_+(y)\subset \mathbb{P}S^n$, the $12$-term localization sequence \cite{Ba00}, together with the d\'evissage theorem \cite{Gille07}, yields the following $12$-term periodic exact sequence
\begin{equation}\label{eq:balmer-12-term-exact}
	\Small
	\xymatrix@C=1.5em@R=3em{
		\W^{3}(S^{n}) \ar[r]^-{\partial_L} & \W^{3}(\mathbb{P}C^{n}, L(1)) \ar[r]^-{i_*} & \W^{0}(\mathbb{P}S^{n}, L) \ar[r]^-{j^*} & \W^{0}(S^{n}) \ar[d]^-{\partial_L} \\
		\W^{3}(\mathbb{P}S^{n}, L) \ar[u]^-{j^*} & & & \W^{0}(\mathbb{P}C^{n}, L(1)) \ar[d]^-{i_*} \\
		\W^{2}(\mathbb{P}C^{n}, L(1)) \ar[u]^-{i_*} & & & \W^{1}(\mathbb{P}S^{n}, L) \ar[d]^-{j^*} \\
		\W^{2}(S^{n}) \ar[u]^-{\partial_L} & \W^{2}(\mathbb{P}S^{n}, L) \ar[l]_-{j^*} & \W^{1}(\mathbb{P}C^{n}, L(1)) \ar[l]_-{i_*} & \W^{1}(S^{n}) \ar[l]_-{\partial_L}
	}
\end{equation}
where $L$ may be taken to be either $\SO$ or $\SO(1)$.

\begin{remark}
	In \cite{DF08}, the localization sequence \eqref{eq:balmer-12-term-exact} is used only for the untwisted case $L=\SO$. In our situation, the twisted case $L=\SO(1)$ is equally important. Although the formal shape of the sequences are similar, the groups and boundary maps behave quite differently.\ This additional input is essential for computing $\W^i(S^n)$.
\end{remark}

The computations in \cite{X19,X26} can be used to reveal the groups $\W^i(\mathbb{P}S^n,L)$ and $\W^i(\mathbb{P}C^n,L(1))$, which we investigate in Section \ref{sect:witt-quadric}. Combining these computations with the exact sequence \eqref{eq:balmer-12-term-exact} for both $L=\SO$ and $L=\SO(1)$, we determine all groups $\W^i(S^n)$ except in the case $i=0$ and $n\equiv 0 \pmod 4$. 

To deal with this delicate remaining case, we show that for $L=\SO(1)$ and $n=4m$, the localization sequence \eqref{eq:balmer-12-term-exact} yields a short exact sequence
\begin{equation}\label{eq:short-exact-WS4}
	\xymatrix{
		0 \ar[r] &
		\W(\mathbb{P}S^{4m}, \SO(1)) \ar[r]^-{j^*} &
		\W(S^{4m}) \ar[r]^-{\partial_{\SO(1)}} &
		\W(\mathbb{P}C^{4m}) \ar[r] &
		0.}
\end{equation}
Moreover, we prove that $\W(\mathbb{P}S^{4m},\SO(1)) \cong \Z \oplus \Z$
and  $\W(\mathbb{P}C^{4m}) \cong \Z/2^{\delta(4m+1)}$ (see Section \ref{sect:witt-quadric}). However, these information do not determine the group structure of $\W(S^{4m})$ completely. The key additional input comes from spinor bundles. More precisely, we construct symmetric forms $(\mathcal U^+,\theta^+)$ and  $(\mathcal U^-,\theta^-)$ 
on the two spinor bundles over $\mathbb{P}S^{4m}$ and show that their classes generate $\W(\mathbb{P}S^{4m},\SO(1)) \cong \Z\oplus \Z.$
We then prove in Lemma \ref{lem:spinor-decompose} that, after restriction to $S^{4m}$, one has an isometry
\[
j^*(\mathcal U^+,\theta^+) \perp j^*(\mathcal U^-,\theta^-)
\cong 2^{\delta(4m+1)}\langle 1\rangle .
\]
This relation is the final piece that allows us to solve the extension problem in \eqref{eq:short-exact-WS4}, and hence to conclude that
$\W(S^{4m}) \cong \Z \oplus \Z$.
	\section{Witt groups of Clifford algebras over $\mathbb{R}$} \label{sec:Clifford}
	
	Let $\varphi_{p,q}$ denote the quadratic form $p \langle -1 \rangle \perp q\langle 1 \rangle $ ($p,q >0$) over $\mathbb{R}$. We adopt the notation $\mathrm{C}^{p,q}: = C(\varphi_{p,q})$ and $\mathrm{C}^{0,0} = \mathbb{R}$ to  denote the Clifford algebra of $\varphi_{p,q}$. Let $\mathrm{C}^{p,q}_0$ be the even part of the Clifford algebra $\mathrm{C}^{p,q}$. 
	We need to use two kinds of involution on $\mathrm{C}_0^{p,q}$, the first one is called the canonical involution (or reversing involution) denoted by $\sigma_0$, and the other one $\sigma_1^v$ is given by $\mathrm{int}(v) \sigma_0 :=  v\sigma_0v^{-1}$ where $v \in \mathbb{R}^{p+q} \subset \C^{p,q}$ is chosen so that $\varphi(v) \neq 0$.\ The type of $\sigma_1^v$ does not depend on the choice of $v$.\ The following result should be well-known to experts. 
	
	\begin{lem}
The structure of the Clifford algebras $\mathrm{C}_0^{0,n}$ and $\mathrm{C}_0^{1,n-1}$ together with the type of the involution $\sigma_0$ and $\sigma_1^v$ are determined in the following table.  
	\end{lem}
	\begin{table}[ht]
		\centering
		\footnotesize
		\renewcommand{\arraystretch}{1.5}
		\scalebox{0.9}{
			\Rotatebox{0}{	\begin{tabular}{|l||c|c||c|c|c|c| c| c|c|}
					\hline
					$n$	   &	{$\mathrm{C}_0^{0,n}$} & 
					{$\mathrm{C}_0^{1,n-1}$} & {Type of $\sigma_{0}$} & {Type of $\sigma_1^v$} \\
					\hline \hline
					$ 1 $	& $\mathbb{R}$ &  $\mathbb{R}$ & orthogonal & orthogonal \\
					\hline
					$2$	& $\mathbb{C} $  & $\mathbb{R} \times \mathbb{R} $ & unitary&orthogonal  \\
					\hline
					$3$	& $\mathbb{H} $ & $M_2(\mathbb{R}) $ & symplectic &orthogonal  \\
					\hline
					$4$	& $\mathbb{H}  \times \mathbb{H} $  & $ M_2(\mathbb{C})$  & symplectic &unitary  \\
					\hline
					$5$	& $M_2(\mathbb{H}) $ & $M_2(\mathbb{H}) $   &  symplectic &symplectic  \\
					\hline
					$6$	& $M_4(\mathbb{C})$  & $M_2(\mathbb{H}) \times M_2(\mathbb{H})$ &  unitary &symplectic  \\
					\hline
					$ 7$	& $M_8(\mathbb{R}) $  & $M_4(\mathbb{H}) $ & orthogonal &symplectic  \\
					\hline
					$ 8$	& $M_8(\mathbb{R}) \times M_8(\mathbb{R})$  & $ M_8(\mathbb{C})$ & orthogonal &unitary  \\
					\hline
		\end{tabular} }}  
		\label{table:clifford-algebra-structure}
	\end{table}

	\begin{proof}
		If $\varphi$ is a quadratic form, then $\C_0(\varphi) \cong \C_0(-\varphi)$ and $\C_0(\varphi \perp \langle  -1 \rangle) \cong C(\varphi)$ \cite[Lemma 4.5]{Swan}. It follows that $\C_0^{0,n} \cong \C^{n,0}_0 \cong \C^{n-1,0}$, and that $\C^{1,n-1}_0 \cong \C^{0,n-1} $. Both algebras are computed in \cite{ABS63}. The types of involutions are settled in the  book of Knus-Merkerjev-Rost-Tignol, \cite{KMRT98}. 
	\end{proof}
	\begin{remark}
		For $n>9$, we may adopt the periocity theorem on Clifford algebras \cite[Section 4]{ABS63}. The type of involution is also periodic.  
	\end{remark}
\begin{theo}\label{thm:Wittcliff} For any integer $n \ge 0$, the Witt groups $\W^i(\C_0^{0,n},\sigma_0)$, $\W^i(\C_0^{0,n},\sigma_1^v)$, $\W^i(\C_0^{1,n-1},\sigma_0)$ and $\W^i(\C_0^{1,n-1},\sigma_1^v)$ are determined in Table \ref{table:sigma_combined} for $i$ even, and vanish for $i$ is odd. 
\end{theo}	
	\begin{table}[H]
		\footnotesize \centering
		\renewcommand{\arraystretch}{1.3}
		\begin{tabular}{|l||c|c|c|c||c|c|c|c|}
			\hline
			\multirow{2}{*}{$n $} & \multicolumn{2}{c|}{$\W^i(\C_0^{0,n}, \sigma_0)$} & \multicolumn{2}{c||}{$\W^i(\C_0^{1,n-1}, \sigma_0)$} & \multicolumn{2}{c|}{$\W^i(\C_0^{0,n}, \sigma_1^v)$} & \multicolumn{2}{c|}{$\W^i(\C_0^{1,n-1}, \sigma_1^v)$} \\
			\cline{2-9}
			& $i = 0$ & $i = 2$ & $i = 0$ & $i = 2$ & $i = 0$ & $i = 2$ & $i = 0$ & $i = 2$ \\
			\hline \hline
			$1$ & $\Z$   & $0$        & $\Z$   & $0$    & $\Z$   & $0$    & $\Z$   & $0$        \\
			\hline
			$2$ & $\Z$   & $\Z$       & $0$    & $0$    & $\Z_2$ & $0$    & $\Z^2$ & $0$        \\
			\hline
			$3$ & $\Z$   & $\Z_2$     & $0$    & $\Z$   & $\Z_2$ & $\Z$   & $\Z$   & $0$        \\
			\hline
			$4$ & $\Z^2$ & $\Z_2^2$ & $0$    & $\Z_2$ & $0$    & $0$    & $\Z$   & $\Z$       \\
			\hline
			$5$ & $\Z$   & $\Z_2$     & $\Z$   & $\Z_2$ & $\Z$   & $\Z_2$ & $\Z$   & $\Z_2$     \\
			\hline
			$6$ & $\Z$   & $\Z$       & $0$    & $0$    & $0$    & $\Z_2$ & $\Z^2$ & $\Z_2^2$ \\
			\hline
			$7$ & $\Z$   & $0$        & $\Z_2$ & $\Z$   & $0$    & $\Z$   & $\Z$   & $\Z_2$     \\
			\hline
			$8$ & $\Z^2$ & $0$        & $\Z_2$ & $0$    & $0$    & $0$    & $\Z$   & $\Z$       \\
			\hline
		\end{tabular}
		\caption{Witt groups of Clifford algebras}
		\label{table:sigma_combined}
	\end{table}
	\begin{proof}
		If $i$ is odd, the result follows from \cite[Theorem 5.2]{DR87}. If $i$ is even, the results follow from Morita equivalence (cf. \cite[Section I.8]{Knus}), together with the following basic computations:
 	\[ \W^i(M_n(\mathbb{H}) \times M_n(\mathbb{H}), \sigma_{uni}) =	\W^i( M_n(\R) \times M_n(\R), \sigma_{uni}) = 0  \]
 	and
	\[ \W^i(\mathbb{H}, \sigma_{sym}) =	\begin{cases}
			\Z & \textnormal{if $i =0$} \\
			\Z_2 & \textnormal{if $i=2$}
		\end{cases} 
		\quad   \quad 
		\W^i(\mathbb{C}, \sigma_{uni}) =	\begin{cases}
		\Z & \textnormal{if $i =0$} \\
		\Z_2 & \textnormal{if $i=2$}
		\end{cases} \]
	\[ \W^{i}(\mathbb{C}, \sigma_{ort}) = \W^{i-2}(\mathbb{C},\sigma_{sym})	= \begin{cases}
	\Z_2 & \textnormal{if $i =0$} \\
    0	 & \textnormal{if $i=2$}
\end{cases}  \quad \quad 
\W^i(\mathbb{R}) = \begin{cases}
	\mathbb{Z} & \textrm{if $i = 0$} \\
	0  & \textnormal{if $i =2 $}
\end{cases}
\]
These computations were done in \cite[Section I.9]{Knus}.  
	\end{proof}
	
	\section{Witt groups of real projective quadrics}\label{sect:witt-quadric}
	This section is devoted to settle the computation of Witt groups of the projective quadrics $\mathbb{P}C^n$ and $\mathbb{P}S^n$ with coefficients in both trivial and twisted line bundles.
	\begin{theo}\label{thm:Psphere}
		The Witt groups $\W^i(\mathbb{P}S^n, \mathcal{O}(j))$ are settled in Table \ref{table:PSn_minus}. 
	\end{theo}
\begin{table}[H]
	\footnotesize \centering
	\renewcommand{\arraystretch}{1.5}
	\begin{tabular}{|l||c|c|c|c||c|c|c|c|}
		\hline
		\multirow{2}{*}{$\bar{n}$} & \multicolumn{4}{c||}{$\W^i(\mathbb{P}S^n)$} & \multicolumn{4}{c|}{$\W^i(\mathbb{P}S^n, \SO(1))$} \\
		\cline{2-9}
		& $i = 0$ & $i = 1$ & $i = 2$ & $i = 3$ & $i = 0$ & $i = 1$ & $i = 2$ & $i = 3$ \\
		\hline \hline
		$1$ & $\Z$ & $\Z$ & $0$ & $0$ & $\Z$ & $\Z$ & $0$ & $0$ \\
		\hline
		$2$ & $\Z$ & $\Z_2$ & $\Z$ & $0$ & $\Z$ & $0$ & $\Z$ & $0$ \\
		\hline
		$3$ & $\Z$ & $\Z_2$ & $0$ & $\Z$ & $\Z$ & $0$ & $\Z_2$ & $\Z$ \\
		\hline
		$4$ & $\Z^2$ & $0$ & $0$ & $0$ & $\Z^2$ & $0$ & $\Z_2^2$ & $0$ \\
		\hline
		$5$ & $\Z$ & $\Z$ & $0$ & $\Z_2$ & $\Z$ & $\Z$ & $\Z_2$ & $0$ \\
		\hline
		$6$ & $\Z$ & $0$ & $\Z$ & $\Z_2$ & $\Z$ & $0$ & $\Z$ & $0$ \\
		\hline
		$7$ & $\Z$ & $0$ & $0$ & $\Z$ & $\Z$ & $0$ & $0$ & $\Z$ \\
		\hline
		$8$ & $\Z^2$ & $0$ & $0$ & $0$ & $\Z^2$ & $0$ & $0$ & $0$ \\
		\hline
	\end{tabular}
	\caption{Witt groups of the quadric $\mathbb{P}S^n$}
	\label{table:PSn_minus}
\end{table}

\begin{proof}
	Since the projective quadric $\mathbb{P}S^n$ is isotropic, one can use \cite[Proposition 8.5]{X19} for the untwisted case $L = \SO$. To be more precise, if $n$ is odd, then there is an isomorphism $\W^i(\mathbb{P}S^{n}) \cong \W^i(\mathbb{R}) \oplus \W^{i+1}(\C_0^{1,n+1},\sigma_0)$.
	If $n$ is even, then we have a long exact sequence 						\[\xymatrix{ \ldots \ar[r] & \W^{i}(\mathbb{R}) \oplus \W^{i+1}(\C_0^{1,n+1},\sigma_0) \ar[r] & \W^i(\mathbb{P}S^n) \ar[r]^-{p_*} & \W^{i-n}(\mathbb{R}) \ar[r] & \ldots }\]
	For the twisted case $L= \SO(1)$, we shall use \cite{X26}. If $n$ is even, then we have an isomorphism $\W^i(\mathbb{P}S^n,\SO(1)) \cong \W^i(\C_0^{1,n+1},\sigma_1^v)$.
	 If $n$ is odd, then there is a long exact sequence
	\[\xymatrix{ \ldots  \ar[r] & \W^i(\C_0^{1,n+1},\sigma_1^v) \ar[r] &  \W^i(\mathbb{P}S^n,\SO(1))  \ar[r]^-{p_*} & \W^{i-n}(\mathbb{R}) \ar[r] & \ldots }\]
	Note that since $\mathbb{P}S^n$ is isotropic and it has a rational point, and therefore $p_*$ is split surjective. The table follows straightforward from Theorem \ref{thm:Wittcliff}. 
\end{proof}

\begin{theo}\label{thm:Pcone}
The Witt groups $\W^i(\mathbb{P}C^{n-1}, \mathcal{O}(j))$ are settled in Table \ref{table:PCn}.

\end{theo}
\begin{table}[H]
	\footnotesize \centering
	\renewcommand{\arraystretch}{1.5}
	\begin{tabular}{|l||c|c|c|c||c|c|c|c|}
		\hline
		\multirow{2}{*}{$\bar{n}$} & \multicolumn{4}{c||}{$\W^i(\mathbb{P}C^{n-1})$} & \multicolumn{4}{c|}{$\W^i(\mathbb{P}C^{n-1}, \SO(1))$} \\
		\cline{2-9}
		& $i = 0$ & $i = 1$ & $i = 2$ & $i = 3$ & $i = 0$ & $i = 1$ & $i = 2$ & $i = 3$ \\
		\hline \hline
		$1$ & $\Z_{2^{\delta(n)}}$ & $0$ & $0$ & $0$ & $\Z_{2^{\kappa(n)}}$ & $0$ & $0$ & $0$ \\
		\hline
		$2$ & $\Z_{2^{\delta(n)}}$ & $\Z_{2^{\kappa(n)}}$ & $0$ & $0$ & $\Z_2$ & $0$ & $0$ & $0$ \\
		\hline
		$3$ & $\Z_{2^{\delta(n)}}$ & $\Z_2$ & $0$ & $0$ & $\Z_2$ & $0$ & $\Z_{2^{\kappa(n)}}$ & $0$ \\
		\hline
		$4$ & $\Z_{2^{\delta(n)}}$ & $\Z_2^2$ & $0$ & $\Z_{2^{\kappa(n)}}$ & $0$ & $0$ & $0$ & $0$ \\
		\hline
		$5$ & $\Z_{2^{\delta(n)}}$ & $\Z_2$ & $0$ & $0$ & $\Z_{2^{\kappa(n)}}$ & $0$ & $\Z_2$ & $0$ \\
		\hline
		$6$ & $\Z_{2^{\delta(n)}}$ & $\Z_{2^{\kappa(n)}}$ & $0$ & $0$ & $0$ & $0$ & $\Z_2$ & $0$ \\
		\hline
		$7$ & $\Z_{2^{\delta(n)}}$ & $0$ & $0$ & $0$ & $0$ & $0$ & $\Z_{2^{\kappa(n)}}$ & $0$ \\
		\hline
		$8$ & $\Z_{2^{\delta(n)}}$ & $0$ & $0$ & $\Z_{2^{\kappa(n)}}$ & $0$ & $0$ & $0$ & $0$ \\
		\hline
	\end{tabular}
	\caption{Witt groups of the quadric $\mathbb{P}C^{n-1}$}
	\label{table:PCn}
\end{table}
Here, the number $\kappa (n)$ is an integer depends on $n$.
\begin{proof}
	The untwisted case $\W^i(\mathbb{P}C^{n-1})$ is proved in \cite[Theorem 1.5]{X19} without the determination of the number $\kappa(n)$. Note that our new notation $\mathbb{P}C^{n-1}$ coincides with the scheme $Q_{0,n}$ in \cite{X19}. The twisted case shall follow from \cite{X26}. If $n$ is even,  we have an isomorphism $\W^i(\mathbb{P}C^{n-1},\SO(1)) \cong \W^i(\C_0^{0,n},\sigma_1^v).$ 	If $n$ is odd, then we use the exact sequence
	\[\xymatrix{ \ldots \ar[r] & \W^i(\C_0^{0,n},\sigma_1^v) \ar[r] &  \W^i(\mathbb{P}C^{n-1},\SO(1))  \ar[r]^-{p_*} & \W^{i-n+2}(\mathbb{R}) \ar[r] & \ldots }\]
Another key input is that the Witt groups $\W^i(\mathbb{P}C^n, \SO(j))$ are all two primary torsion \cite{DF08}. These data suffices to obtain the result by Theorem \ref{thm:Wittcliff}. To explain, if $n =3$, we obtain an exact sequence
\[ \xymatrix{0 \ar[r] & \W^1(\mathbb{P}C^2, \SO(1)) \ar[r] & \Z \ar[r] & \Z \ar[r] & \W^2(\mathbb{P}C^2,\SO(1)) \ar[r] & 0 }\]
Since $\W^1(\mathbb{P}C^2, \SO(1)) $ is two primary torsion, this exact sequence  forces the group $\W^1(\mathbb{P}C^2, \SO(1))$ vanishes.  The other situation is similar whenever the number $\kappa(n)$ shows up in the table. 
\end{proof}

	\section{Witt groups of the real sphere}\label{sec:Wittsphere}	
	\begin{theo}\label{thm:Wittsphere}
		The group structure of $\W^i(S^n)$ is settled in Table \ref{table:S_n}.  
		\begin{table}[htbp]
			\Small 
			\centering
			\renewcommand{\arraystretch}{1.5}
			\begin{tabular}{|c||c|c|c|c|c|c|c|c|c|c|}
				\hline
				$\bar{n}$ & $1$ & $2$ & $3$ & $4$ & $5$ & $6$ & $7$ & $8$ \\
				\hline  \hline
				$\W^0(S^n)$ & $\mathbb{Z} \oplus \mathbb{Z}_2$ & $\mathbb{Z} \oplus \mathbb{Z}_2$ & $ \quad \mathbb{Z} \quad $  & $\mathbb{Z} \oplus \mathbb{Z} $ & $ \quad \mathbb{Z} \quad $ & $ \quad \mathbb{Z} \quad $ & $ \quad \mathbb{Z} \quad $ & $\mathbb{Z} \oplus \mathbb{Z}$ \\
					\hline 
				$W^1(S^{{n}})$ & $\Z$ & $\Z_{2}$ & $\Z_{2}$ & $0$ & $\Z$ & $0$ & $0$ & $0$  \\
					\hline 
				$W^2(S^{{n}})$ & $0$ & $\Z$ & $0$ & $\Z_{2}$ & $0$ & $\Z$ & $0$ & $0$  \\
					\hline 
				$W^3(S^{{n}})$ & $0$ & $0$ & $\Z$ & $0$ & $0$ & $0$ & $\Z$ & $0$  \\
				\hline
			\end{tabular}
			\caption{Witt groups of the real sphere}
			\label{table:S_n}
		\end{table}
	\end{theo}

	\begin{proof}
Note that $W^i(S^n)$ appears in the localization sequence \eqref{eq:balmer-12-term-exact} in both the untwisted case $L=\SO$ and the twisted case $L=\SO(1)$. By comparing these two parallel localization sequences and applying Theorems \ref{thm:Psphere} and \ref{thm:Pcone}, Table \ref{table:S_n} follows immediately except in the case $i=0$ and $n \equiv 0 \mod 4$. 

Now assume that $n=4m$. Then, Theorems \ref{thm:Psphere} and \ref{thm:Pcone} yield the exact sequence \eqref{eq:short-exact-WS4} without difficulty. It remains to resolve the extension problem arising from \eqref{eq:short-exact-WS4}. Note that the symmetric form $(\UC_a, \theta)$ on the spinor bundle (cf. \cite[Prop. 4.7]{X26}) over $\mathbb{P}S^{n}$ with coefficients in $\SO(1)$ induces an isomorphism 
\[ (\UC_a, \theta)\otimes - : W\big(\C_0^{1,n+1},\sigma_1^v\big) \to W(\mathbb{P}S^n,\SO(1))  \]
by \cite[Proposition 5.9]{X26}. By an appropriate choice of $v$, we have an isomorphism $(\C^{1,n+1}_0,\sigma_1^v) \cong (\C^{0,n+1},\sigma_0) $
of algebras with involution. Therefore, we may concentrate on the Clifford algebra $\C^{0,n+1}$. 

Now, we describe the generators of $W(\mathbb{P}S^n,\SO(1))$.\ If $n \equiv 0 \mod 4$, the element $w := e_0 e_1 \ldots  e_n$ with $w^2 =1$ is in the center $Z(\C^{0,n+1})$, which is isomorphic to $\R \times \R$. Define $\hbar_\pm: = \frac{1\pm w}{2} \in Z(\C^{0,n+1})$, and it is not hard to see that $\hbar_\pm^2 = \hbar_\pm$, $\hbar_+ + \hbar_- = 1$ and $\hbar_+ \hbar_- = 0$. Therefore, 
\[ \C^{0,n+1} \cong \C^{0,n+1} \hbar_+ \oplus  \C^{0,n+1} \hbar_- \]
with $\C^{0,n+1} \hbar_\pm$ two-sided ideals of $\C^{0,n+1}$. Note that $\C^{0,n+1}\hbar_\pm$ are both isomorphic to $\C^{n,0}$ which is a simple algebra. 
Consider the algebra automorphism $\tau:\C^{0,n+1}\to \C^{0,n+1}$ with $ \tau(e_i)=-e_i$, with $0\le i\le n$.
Since $n+1$ is odd, we have $\tau(w)=-w$, hence $\tau(\hbar_+)=\hbar_-$.
Choose a primitive idempotent $\wp_+\in \C^{0,n+1}\hbar_+$ and set $\wp_-:=\tau(\wp_+)\in \C^{0,n+1}\hbar_-.$ Then $\wp_-$ is a primitive idempotent in $\C^{0,n+1}\hbar_-$. 

Under the isomorphism $(\C^{1,n+1}_0,\sigma_1^v) \cong (\C^{0,n+1},\sigma_0) $, tensoring the symmetric bundle $(\UC_a,\theta)$ in \cite[Prop 4.7]{X26} with the two irreducible left $\C^{0,n+1}$-ideals $E_\pm:= \C^{0,n+1} \wp_\pm$ equipped with the symmetric forms induced by the trace forms, yields two parallel symmetric bundles $(\UC^{\pm}, \theta^\pm)$, which represent the two generators of $W(\mathbb{P}S^n,\SO(1))$. It is straightforward to check that 
\[j^*(\UC^{\pm}, \theta^\pm) \cong (P_n^\pm,\hat\beta|_{P_n^\pm}) \]
where  we define 
\[P^\pm_n:= \ker(\Phi: A_n \otimes_{\mathbb{R}} E_\pm \to A_n \otimes_{\mathbb{R}} E_\pm ). \]

\begin{lem}\label{lem:spinor-decompose}
The orthogonal sum
\[(P_n^+,\hat\beta|_{P_n^+})\perp (P_n^-,\hat\beta|_{P_n^-})\]
is isometric to the trivial form $2^{\delta(n+1)} \langle 1 \rangle$ over $S^n$.
\end{lem}

\begin{proof}
Set $M_\pm:=A_n\otimes_{\mathbb R}E_\pm.$
Since $\tau(\wp_+)=\wp_-$, the map
\[ T:M_-\longrightarrow M_+,
\quad
T(a\otimes u):=a\otimes \tau(u), \]
is an $A_n$-module isomorphism. We first check that $T$ is an isometry if we input the form $\hat\beta$ over both sides of $T$. Note that we have
$\tr(\tau(z))=\tr(z)$
for all  $z\in \C^{0,n+1}$.
Since $\tau$ commutes with the canonical involution $\sigma$, we have
\[\beta(\tau(u),\tau(v))
= \tr \bigl(\sigma(\tau(u))\,\tau(v)\bigr)
= \tr \bigl(\tau(\sigma(u)v)\bigr)
= \tr(\sigma(u)v)
= \beta(u,v).\]
for $u,v\in E_+$. Therefore, $T$ is an isometry.

Furthermore, since $\tau(e_i)=-e_i$ for all $i$, we have
\[T(\mathbf{x}\cdot (a\otimes u))
=-\mathbf{x}\cdot T(a\otimes u).\]
Thus, $ T\circ \Phi=(1-\Phi)\circ T.$
Consequently,
$T(\Ker(\Phi))=\Ker(1-\Phi),$
that is,
\[T(P_n^-)=Q_n^+:=\Ker(1-\Phi:M_+\to M_+) \]
which yields an isometry
\[(P_n^-,\hat\beta|_{P_n^-})\cong (Q_n^+,\hat\beta|_{Q_n^+}).\]
The proof of Lemma \ref{lem:Pntheta} applied to $(M_+, \hat\beta)$ gives that
\[(P_n^+,\hat\beta|_{P_n^+})\perp (P_n^-,\hat\beta|_{P_n^-})
\cong (P_n^+,\hat\beta|_{P_n^+})\perp (Q_n^+,\hat\beta|_{Q_n^+})
\cong (M_+,\hat\beta).\]
Finally, $\beta$ is positive definite on $E_+$ which must be the trivial form. If follows that
$(M_+,\hat\beta)$
is the trivial form of rank $\dim(E_+)=2^{\delta(n+1)}$.
This proves the lemma.
\end{proof}
The result that $W(S^n) \cong \Z \oplus \Z$ for $n \equiv 0 \mod 4$ follows.  \end{proof}

	\begin{ackno}I would like to thank Marco Schlichting for his encouragement and helpful discussions. I also want to thank Baptiste Calm\`{e}s,
		Jean Fasel, 
		Jens Hornbostel, 
		Max Karoubi 
		and Chuck Weibel
		for useful comments.\ 
	\end{ackno}

\end{document}